\documentclass[10pt]{amsart}

\usepackage{amsmath,xspace,amssymb,mathrsfs}
\usepackage{color}

\input xy
\xyoption{all}
\xyoption{2cell}
\UseAllTwocells
\CompileMatrices

\newcommand{\gr}{\operatorname{gr}}
\newcommand{\Spec}{\operatorname{Spec}}
\renewcommand{\phi}{\varphi}

\newcommand{\Ker}{\operatorname{Ker}}

\newcommand{\Ima}{\operatorname{Im}}

\newcommand{\Max}{\operatorname{Max}}

\newcommand{\Ann}{\operatorname{Ann}}

\newcommand{\Supp}{\operatorname{Supp}}

\newcommand{\dg}{\operatorname{deg}}

\newtheorem{proposition}{Proposition}[section]
\newtheorem{lemma}[proposition]{Lemma}

\newtheorem{corollary}[proposition]{Corollary}
\newtheorem{theorem}[proposition]{Theorem}

\newtheorem{conjecture}[proposition]{Conjecture}

\theoremstyle{definition}

\newtheorem{example}[proposition]{Example}

\newtheorem{remark}[proposition]{Remark}

\usepackage{etoolbox}
\makeatletter
\patchcmd{\@settitle}{\uppercasenonmath\@title}{}{}{}
\patchcmd{\@setauthors}{\MakeUppercase}{}{}{}
\makeatother

\begin{document}

\title[Zero-divisors, units and colon ideals]{Homogeneity of zero-divisors, units and idempotents in a graded ring}

\author[A. Tarizadeh]{Abolfazl Tarizadeh}
\address{Department of Mathematics, Faculty of Basic Sciences, University of Maragheh, Maragheh, East Azerbaijan Province, Iran.}
\email{ebulfez1978@gmail.com}

\date{}

\begin{abstract} In this article we prove several important results on graded rings, especially monoid-rings, that are motivated and inspired by Kaplansky's zero-divisor, unit and idempotents conjectures. Among the main results, we first generalize Kaplansky's zero-divisor conjecture of group-rings $K[G]$ (with $K$ a field) to the more general setting of $G$-graded rings $R=\bigoplus\limits_{n\in G}R_{n}$ with $G$ a torsion-free group. Then we prove that if $I$ is an unfaithful left ideal of a $G$-graded ring $R$ with $G$ a totally ordered group, then there exists a (nonzero) homogeneous element $g\in R$ such that $gI=0$. 
This theorem gives an affirmative answer to the new conjecture in the case that the group involved in the grading is a totally ordered group. 
Our result also generalizes McCoy's famous theorem on polynomial rings to the more general setting of $G$-graded rings. Next, we focus on Kaplansky's unit conjecture. In this regard, we completely characterize the units (invertible elements) of a $G$-graded ring with $G$ a torsion-free Abelian group. Then we prove a key result which asserts that every nonzero idempotent element of a $G$-graded commutative ring with $G$ a torsion-free Abelian group is homogeneous of degree zero. This important special case leads us to derive a general result which asserts that every idempotent element of a $G$-graded commutative ring $R=\bigoplus\limits_{n\in G}R_{n}$ with $G$ an Abelian group is contained in the subring $\bigoplus\limits_{h\in H}R_{h}$ where $H$ is the torsion subgroup of $G$. This theorem gives an affirmative answer to the generalized version of Kaplansky's idempotent conjecture in the commutative case. Next, we prove that the Jacobson radical of every $G$-graded commutative ring with $G$ a torsion-free Abelian group is a graded ideal. This result generalizes Bergman’s theorem in the commutative case (which asserts that the Jacobson radical of every $\mathbb{Z}$-graded commutative ring is a graded ideal). Finally, we show that an Abelian group $G$ is a torsion-free group if and only if the Jacobson radical of every $G$-graded ring is a graded ideal, or equivalently, the idempotents of every $G$-graded ring are contained in the base subring, or equivalently, the group-ring $\mathbb{Q}[G]$ has no nontrivial idempotents.
\end{abstract}

\maketitle

\section{Introduction}

Kaplansky's zero-divisor conjecture claims that if $G$ is a torsion-free group (not necessarily Abelian) and $K$ is a field, then the only zero-divisor element of the group-ring $K[G]$ is zero. Despite the proof of some special cases in the literature (see e.g. \cite{Camillo},
\cite{Cliff}, \cite{Kropholler-et-al}, \cite{McCoy}, \cite{McCoy II} and \cite[Lemma 1.9 on p. 589]{Passman}), this conjecture is still open in general. In order to get deeper insight into the above conjecture, we first generalize it as follows:   

\begin{conjecture} \label{Conjecture I} For any ring $R$ and a torsion-free group $G$, if $f$ is a zero-divisor element of the group-ring $R[G]$ then there exists a nonzero $c\in R$ such that $cf=0$. 
\end{conjecture}

The group-ring $R[G]$ is a typical example of $G$-graded rings. Hence, the above conjecture can be further generalized to the more general setting of graded rings and simultaneously to the setting of ideals rather than to that of elements:

\begin{conjecture} (generalized zero-divisor conjecture)\label{Conjecture II} If $I$ is an unfaithful left ideal of a $G$-graded ring $R$ with $G$ a torsion-free group, then there exists a (nonzero) homogeneous element $g \in R$ such that $gI=0$.
\end{conjecture}

In Theorem \ref{th T-O}, we prove the above conjectures (including non-commutative rings) in the case that the group $G$ involved in the grading is a totally ordered group which is not necessarily Abelian. In the proof of this theorem, we use an interesting idea whose special case was already used by McCoy to establish his result \cite{McCoy II}. In fact, in the course of two articles \cite{McCoy} and \cite{McCoy II}, McCoy proved two remarkable results on polynomial rings. His first result \cite[\S3, Theorems 2-3]{McCoy} asserts that every zero-divisor element of the polynomial ring $R[x_{1},\ldots,x_{d}]$ is annihilated by a nonzero element of the commutative ring $R$. Then in the second result \cite[Theorem I]{McCoy II}, he proves the same assertion for every unfaithful ideal of the polynomial ring $R[x_{1},\ldots,x_{d}]$ (note that in this result, the ring $R$ is not necessarily commutative).
Our theorem (Theorem \ref{th T-O}) also generalizes these results to the more general setting of graded rings. McCoy's theorems are then deduced as special cases of our general result (see Corollaries \ref{Coro TM 18} and \ref{Corollary III}).

Kaplansky's unit conjecture claims that if $G$ is a torsion-free group (not necessarily Abelian) and $K$ is a field, then the units of the group-ring $K[G]$ are precisely homogeneous elements of the form $r\epsilon_{x}$ where $0\neq r\in K$ and $x\in G$. Recently Gardam \cite{Gardam} proved that the group-ring $\mathbb{F}_{2}[F(6)]$ with $F(6)$ the 6th Fibonacci group, which is a torsion-free group and virtually abelian, is indeed a counterexample to the unit conjecture (in fact, he discovered a non-homogeneous invertible element in this group-ring with 21 homogeneous components). In light of this, Theorems \ref{Theorem ix} and \ref{coro ET new 20} are the most general positive results one can expect about the generalized version of the unit conjecture. In fact, in Lemma \ref{Theorem V}, we first show that every invertible element of a $G$-graded integral domain with $G$ a totally ordered group is homogeneous. This key result then enables us to give a complete characterization of invertible elements in graded commutative rings.  Theorems \ref{Theorem ix} and \ref{coro ET new 20} , in particular, tell us that the homogeneous components of an invertible element of a $G$-graded commutative ring (with $G$ a torsion-free Abelian group) form a co-maximal ideal and all distinct double products are nilpotent.

Kaplansky's idempotent conjecture claims that if $G$ is a torsion-free group (not necessarily Abelian) and $K$ is a field, then the group-ring $K[G]$ has no nontrivial idempotents. Despite the proof of some special cases (see e.g. \cite{Cliff}, \cite{Cohen-Mongomery}, \cite{Kanwar-Leroy}, \cite{Kirby} and \cite{Passman}), this conjecture, like the zero-divisor conjecture, is still open in general. Similarly above (i.e., zero-divisor conjecture), we first generalize it in the following way.   
 
\begin{conjecture}(generalized idempotent conjecture) Every idempotent of a $G$-graded ring  $R=\bigoplus\limits_{n\in G}R_{n}$ with $G$ a torsion-free group is contained in the base subring $R_{0}$ where $0$ is the identity element of $G$.
\end{conjecture}

The second version of the generalized idempotent conjecture reads as follows:

\begin{conjecture} Every idempotent of a $G$-graded ring  $R=\bigoplus\limits_{n\in G}R_{n}$ with $G$ an Abelian group is contained in the subring $\bigoplus\limits_{n\in H}R_{n}$ where $H$ is the torsion subgroup of $G$.
\end{conjecture}

Our results completely settle down the above conjectures in the commutative case (see Theorems \ref{Lemma T-K}, \ref{Theorem i-T} and \ref{Theorem 3-III}). In fact, in Theorem \ref{Theorem 3-III}, we prove a general result which asserts that if $M$ is a commutative monoid, then every idempotent element of an $M$-graded commutative ring $R=\bigoplus\limits_{m\in M}R_{m}$ is contained in the subring $\bigoplus\limits_{m\in\widetilde{M}}R_{m}$ where we call $\widetilde{M}=\{x\in M : \exists n\geqslant1, \exists y\in M, nx+y=y\}$ the quasi-torsion submonoid of $M$. We prove this theorem in three main steps. In the first step (Theorem \ref{Lemma T-K}), we prove a key result which asserts that every idempotent element of a $G$-graded commutative ring $R=\bigoplus\limits_{n\in G}R_{n}$ with $G$ a torsion-free Abelian group is contained in the base subring $R_{0}$. 
In the second step (Theorem \ref{Theorem i-T}), by reducing the problem to the finitely generated case and then using the fundamental theorem of finitely generated Abelian groups and most importantly  using the first step, we show that every idempotent element of a $G$-graded commutative ring $R=\bigoplus\limits_{n\in G}R_{n}$ with $G$ an Abelian group is contained in the subring $\bigoplus\limits_{h\in H}R_{h}$ where $H$ is the torsion subgroup of $G$. Finally in the third step (Theorem \ref{Theorem 3-III}), by passing to the grading with the Grothendieck group and then using the second step, we complete the proof. As a consequence, we obtain that for any commutative ring $R$ and for any commutative monoid $M$, then every idempotent element of the monoid-ring $R[M]$ is contained in the subring $R[\widetilde{M}]$. 

Next, we propose the following general problem (the group $G$ is not necessarily Abelian, and the ring involved is not necessarily commutative):

\begin{conjecture} The Jacobson radical of every $G$-graded ring with $G$ a torsion-free group is a graded ideal.
\end{conjecture}

Then we prove the above conjecture in the commutative case (see Theorem \ref{Theorem II}). This theorem considerably generalizes Bergman's theorem \cite[Corollary 2]{Bergman} which was already proved in the literature by difficult methods (see e.g. \cite{Bergman,Cohen-Mongomery,
Ilic,Kelarev,Kelarev-Okininski,Mazurek}). 

In \cite[Lemma 1]{Armendariz}, Armendariz obtained an interesting result on reduced polynomial rings (which is closely related to McCoy's theorem). His result asserts that in a reduced polynomial ring, if a polynomial $f$ annihilates another polynomial $g$ (i.e. $fg=0$) then each coefficient of $f$ annihilates every coefficient of $g$. In Theorem \ref{Theorem vii}, we prove a more general result which asserts that if $I$ is a graded radical ideal of a $G$-graded commutative ring $R$ with $G$ a torsion-free Abelian group and $J$ an arbitrary ideal of $R$, then the colon ideal $I:_{R}J=\{r\in R: rJ\subseteq I\}$ is a graded ideal. 
Our theorem vastly generalizes Armendariz' result on polynomial rings to the more general setting of graded rings and simultaneously to arbitrary (not necessarily reduced) rings. Next, in Theorem \ref{Corollary onbes-fifteen}, we codify some important properties of monoid-rings.  

The above results also allow us to show that an Abelian group $G$ is torsion-free if and only if the Jacobson radical of every $G$-graded ring is a graded ideal, or equivalently, the idempotents of every $G$-graded ring are contained in the base subring, or equivalently, the group-ring $\mathbb{Q}[G]$ has no nontrivial idempotents (see Theorem \ref{Bergman's th charact}). Finally, we prove that if a $G$-graded commutative ring $R$ with $G$ a torsion-free Abelian group has a non-zero-divisor homogeneous element of nonzero degree, then the nilradical and the Jacobson radical of $R$ are the same (see Theorem \ref{Lemma 3 NZD} and Corollaries \ref{Coro 2 Gro}, \ref{Coro 3 nice}). This result also generalizes several important results in the literature. All of the results of this article are supplemented with examples showing that some conditions on the grading group $G$ is required. In fact, these examples heavily depend on torsion which highlight the necessity of the torsion-free assumption. In proving the main results of this article, various ideas and methods that were used are: the Grothendieck group, changing the grading, homogeneity of invertible elements of graded domains, lifting idempotents modulo the nilradical, using the whole strength of totally ordered groups and the trick of reducing problem to the finitely generated case.

\section{Preliminaries}\label{Sec:Prelim}

In this section, we recall some necessary background, for the convenience of the reader. The subject of this article is mainly commutative algebra, but some results also apply to non-commutative rings (see e.g. Theorem \ref{th T-O}).

Recall that the Grothendieck group of a commutative monoid $M$ is constructed in the following way. We first define a relation over the set $M\times M$ as $(a,b)\sim(c,d)$ if there exists some $m\in M$ such that $(a+d)+m=(b+c)+m$. It can be seen that it is an equivalence relation. Here we denote the equivalence class containing of an ordered pair $(a,b)$ simply by $[a,b]$, and we denote by $G=\{[a,b]: a,b\in M\}$  the set of all equivalence classes obtained by this relation. Then the set $G$ by the operation $[a,b]+[c,d]=[a+c,b+d]$ is an Abelian group. Indeed, $[0,0]$ is the identity element of $G$ where $0$ is the identity of $M$, and for each $[a,b]\in G$ its inverse is $[b,a]$. This group $G$ is called the \emph{Grothendieck group of $M$}. Note that in the above construction, the commutativity of the monoid $M$ plays a vital role. We have a canonical morphism of monoids $M\rightarrow G$  which is given by $m\mapsto[m,0]$. This map is injective if and only if $M$ has the cancellation property. For example, the Grothendieck group of the additive monoid of natural numbers $\mathbb{N}=\{0,1,2,\ldots\}$ is called the additive group of integers and is denoted by $\mathbb{Z}$. 

Let $R=\bigoplus\limits_{m\in M}R_{m}$ be an $M$-graded ring with $M$ a monoid.  Every nonzero element $0\neq r\in R_{m}$ is called a homogeneous element of $R$ of degree $m\in M$ and written $\dg(r)=m$. If $r,r'\in R$ are homogeneous elements with $rr'\neq0$, then $rr'$ is homogeneous and $\dg(rr')=\dg(r)+\dg(r')$.
Whenever we write $f= \sum\limits_{m\in M} r_{m} \in R$, if it is not stated, then it should be understood that $r_m\in R_m$ for all $m\in M$. We also recall changing the grading on a given graded ring, as we shall need it later on. Let $\phi:M\rightarrow M'$ be morphism of  monoids. Then $R=\bigoplus\limits_{d\in M'}S_{d}$ can be viewed as an $M'$-graded ring with the homogeneous components $S_{d}=\sum\limits_{\substack{m\in M,\\
\phi(m)=d}}R_{m}$ (it is clear that if $d\in M'$ is not in the image of $\phi$, then $S_{d}=0$).

By a totally (linearly) ordered monoid we mean a monoid $M$ equipped with a total ordering $<$ such that its operation is compatible with its ordering, i.e. if $a<b$ for some $a,b\in M$, then $a+x \leqslant b+x$ for all $x\in M$. If moreover, $M$ has the cancellation property then from $a<b$ we get that $a+x<b+x$.    

Recall that for a given ring $R$, then an element $x$ of an $R$-module $M$ is called a \emph{torsion element} if $ax=0$ for some non-zero-divisor $a\in R$. It can be seen that $T(M)=\{x\in M:\exists a\in R\setminus Z(R), ax=0\}$, the set of torsion elements,  is an $R$-submodule of $M$ (note that $R$ does not need to be an integral domain here) where $Z(R)$ denotes the set of zero-divisors of $R$. The module $T(M)$ is called the \emph{torsion submodule} or the \emph{torsion part of $M$}.
An $R$-module $M$ is called \emph{torsion-free} if $T(M)=0$, and \emph{torsion} if $T(M)=M$. For example, every nonzero ideal of an integral domain is a torsion-free module. 
Every flat module (and hence every free or projective module) is torsion-free. If $A$ is the total ring of fractions of a ring $R$ and $B$ is an extension ring of $A$, then as $R$-modules we have $T(B/R)=A/R$. In particular, $A/R$ is a torsion $R$-module. 
 
If $G$ is an Abelian group then as a $\mathbb{Z}$-module, $T(G)=\{x\in G:\exists n\geqslant1, nx=0\}$ is the torsion subgroup of $G$ (the subgroup of elements of finite order). 

It can be seen that every totally ordered group is torsion-free (every non-identity element is of infinite order). It is very important to note that the reverse implication is also true in the commutative case: an Abelian group is torsion-free if and only if it is a totally ordered group. Its proof can be found in \cite[Theorem 6.31]{Lam} and \cite[\S3]{Levi}. Hence, all of the results (of this article) remain true if one replaces ``totally ordered Abelian group'' by ``torsion-free Abelian group'' (or vice versa). More generally, it can be shown that a commutative monoid $M$ with the cancellation property is a totally ordered monoid if and only if the Grothendieck group of $M$ is a torsion-free group (see \cite[\S 2.12, Theorem 22]{Northcott}).

If $M$ is a finitely generated module over a principal ideal domain $R$, then by the structure theorem over PIDs, we have the factorization $M=T(M)\oplus F$ where $F$ is a free $R$-submodule of $M$ of finite rank (i.e., $F\simeq R^{d}$ as $R$-modules for some natural number $d\geqslant0$). In particular, if $G$ is a finitely generated Abelian group then we have the decomposition $G=T(G)\oplus F$ where $F$ is a finitely generated torsion-free subgroup of $G$. If the annihilator of a module over a ring is a nonzero ideal of the ring, then the module is called unfaithful.

\section{Grading by a torsion-free Abelian group}

In this section, we present some important facts about gradings by torsion-free Abelian groups that will be required in subsequent sections. 

\begin{lemma}\label{Lemma 2-iki} Let $G$ be a torsion-free Abelian group and $I$ a graded proper ideal of a $G$-graded ring $R=\bigoplus\limits_{x\in G}R_{x}$ with the property that whenever $rr'\in I$ for some homogeneous elements $r,r'\in R$, we have $r\in I$ or $r'\in I$. Then $I$ is a prime ideal of $R$.
\end{lemma}

\begin{proof} It is an interesting exercise (and well-known).  
\end{proof} 

If $M$ is a totally ordered monoid (by an ordering $<$) with the cancellation property, then it can be seen that its Grothendieck group is a totally ordered group whose order is defined by $[a,b]<[c,d]$ if $a+d<b+c$ in $M$. In this regard, see also Example \ref{Example 3-uch}. Then the above lemma leads us to the following observation:

\begin{corollary}\label{Coro onalti 16} Let $M$ be a totally ordered cancellative monoid and $I$ a graded proper ideal of an $M$-graded ring $R$ with the property that whenever $rr'\in I$ for some homogeneous elements $r,r'\in R$, we have $r\in I$ or $r'\in I$. Then $I$ is a prime ideal of $R$.
\end{corollary}

\begin{proof} Let $G$ be the Grothendieck group of $M$ which is a torsion-free Abelian group. By changing the grading, using the canonical morphism of monoids $M\rightarrow G$, then $R$ can be viewed as a $G$-graded ring. Since $M$ is cancellative, the above map is injective and so the homogeneous components remain the same with this new grading. 
Now the assertion follows from Lemma~\ref{Lemma 2-iki}. 
\end{proof}

\begin{example}\label{Example 3-uch} Note that even if $M$ is a torsion-free monoid (i.e. every non-identity element is of infinite order) with the cancellation property, then the Grothendieck group of $M$ is not necessarily a torsion-free group (and hence is not necessarily a totally ordered group). For example, let $M$ be the submonoid of the additive group $\mathbb{Z}\times(\mathbb{Z}/n)$ (with $n\geqslant2$) which is generated by $x=(1,0)$ and $y=(1,1)$. Then $M=\mathbb{N}x+\mathbb{N}y$ has the cancellation property, because every submonoid of a group has this property. It can be easily seen that $M$ is torsion-free. Let $G=\{[r,s]: r,s\in M\}$ be the Grothendieck group of $M$. Then we obtain an injective morphism of groups $f:G\rightarrow\mathbb{Z}\times(\mathbb{Z}/n)$ which is given by $[r,s]\mapsto r-s$. Then we show that it is also surjective. Take $(a,b) \in\mathbb{Z}\times(\mathbb{Z}/n)$ then $(a,b)=(a,0)+(0,b)$. If $a\geqslant0$ then $[ax,0] \in G$ is mapped into $(a,0)$. But if $a<0$ then $[0,-ax] \in G$ is mapped into $(a,0)$.  Clearly $[y,x] \in G$ is mapped into $y-x=(0,1)$. Thus if  $b\geqslant0$ then $b[y,x] \in G$ is mapped into $(0,b)$. If $b<0$ then $-b[x,y] \in G$ is mapped into $-b(0,-1)=(0,b)$. But $G\simeq\mathbb{Z}\times(\mathbb{Z}/n)$ is not a torsion-free group.
\end{example}

Recall that if $I$ is an ideal of an $M$-graded ring $R$ (with $M$ a monoid), then by $I^{\ast}$ we mean the ideal of $R$ generated by all homogeneous elements of $R$ contained in $I$. In fact, $I^{\ast}$ is the largest graded ideal of $R$ which is contained in $I$. If $\mathfrak{p}$ is a prime ideal of an $M$-graded ring $R$, then $\mathfrak{p}^{\ast}$ is a graded proper ideal of $R$ and has the above property, i.e. if whenever $ab\in\mathfrak{p}^{\ast}$ for some homogeneous elements $a,b\in R$ then $a\in\mathfrak{p}^{\ast}$ or $b\in\mathfrak{p}^{\ast}$. But it is important to note that $\mathfrak{p}^{\ast}$ is not a prime ideal of $R$ in general (see Example \ref{Example i-nice}). However, in the torsion-free case we have the following result:

\begin{corollary}\label{Coro 20 twenty yirmi} If $\mathfrak{p}$ is a prime ideal of a $G$-graded ring $R$ with $G$ a torsion-free Abelian group, then $\mathfrak{p}^{\ast}$ is a prime ideal of $R$.
\end{corollary}

\begin{proof} It follows from Lemma \ref{Lemma 2-iki}.
\end{proof}

\begin{corollary}\label{Coro 14-ondort} Let $R$ be a $G$-graded ring with $G$ a torsion-free Abelian group. Then every minimal prime ideal of $R$ is a graded ideal. In particular, the nilradical of $R$ is a graded ideal.
\end{corollary}

\begin{proof} It follows from the above corollary. 
\end{proof}

\begin{remark}  The above corollaries can be also generalized to the setting of $M$-graded rings where $M$ is a totally ordered monoid with the cancellation property. 
\end{remark}

\begin{example}\label{Example i-nice} It is important to note that the above results may not be true for gradings by arbitrary Abelian groups. For example, let $R$ be an integral domain of characteristic $p>0$. By changing the grading, we can naturally consider the polynomial ring $R[x]=\bigoplus\limits_{d=0}^{p-1}S_{d}$ as a $\mathbb{Z}/p$-graded ring with the homogeneous components $S_{d}=\sum\limits_{n\geqslant0}Rx^{np+d}$. Note that $\mathbb{Z}/p$, the additive group of integers modulo $p\mathbb{Z}$, is not a torsion-free group. We know that $\mathfrak{p}=(x-1)$ is a prime ideal of $R[x]$, because $R[x]/(x-1)\simeq R$. 
But $\mathfrak{p}^{\ast}$ is not a prime ideal of $R[x]$, because the Frobenius endomorphism gives us $(x-1)^{p}=x^{p}-1\in S_{0}\cap\mathfrak{p}$ and so $(x-1)^{p}\in\mathfrak{p}^{\ast}$ but $x-1\notin\mathfrak{p}^{\ast}$. Also note that $R[x]$ modulo the graded ideal $I=(x^{p}-1)$ is a $\mathbb{Z}/p$-graded ring. The element $(x-1)+I$ of this ring is nilpotent, since $(x-1)^{p}\in I$. But $1\notin I$ (and also $x+I$ is not nilpotent). Thus the nilradical of the $\mathbb{Z}/p$-graded ring $R[x]/I$ is not a graded ideal, and hence it has a minimal prime ideal which is not a graded ideal. In fact, $\mathfrak{p}/I$ is the only minimal prime ideal of this ring which is not a graded ideal. 
\end{example}

The following result is already well known for $\mathbb{N}$-graded and $\mathbb{Z}$-graded rings, we generalize it to more general settings: 

\begin{lemma}\label{Lemma 5-five-bes} Let $R=\bigoplus\limits_{m\in M}R_{m}$ be an $M$-graded ring with $M$ a cancellative monoid. Then $1\in R_{0}$ where $0$ is the identity element of $M$.   
\end{lemma}

\begin{proof} We may write $1=\sum\limits_{m\in M}r_{m}$ where $r_{m}\in R_{m}$ for all $m\in M$. If $x\in M$, then $r_{x}=\sum\limits_{m\in M}r_{x}r_{m}$. Since $M$ has the cancellation property, $r_{x}r_{0}$ is the only homogeneous element of degree $x$ on the right hand side of the above equality. Hence, $r_{x}=r_{x}r_{0}$. This yields that $1=(\sum\limits_{m\in M}r_{m})r_{0}=r_{0}\in R_{0}$.
\end{proof}

Under the assumptions of the above lemma, $R_{0}$ is a subring of $R$. 

\begin{example} It is important to note that if $M$ is not a cancellative monoid, then the unit element of an $M$-graded ring not only is not necessarily of degree zero (= the identity element of $M$), but also not necessarily homogeneous. For example, consider the monoid $M=\{1,3,0\}$, the submonoid of the multiplicative monoid of the ring $\mathbb{Z}/6$. Then the direct product ring $R=\mathbb{Z}\times\mathbb{Z}=R_1\oplus R_3\oplus R_0$ is an $M$-graded ring with the homogeneous components $R_1=\{(0,0)\}$, $R_{3}=\mathbb{Z}\times\{0\}$ and $R_0=\{0\}\times\mathbb{Z}$. The unit of this ring is not homogeneous, because $(1,1)=(1,0)+(0,1)$. As another example, the Nagata idealization ring $R=\mathbb{Z}\times\mathbb{Z}$ is also an $M$-graded ring with the same homogeneous components as above. The unit of this ring $(1,0)$ is homogeneous of degree $3\neq1=$ the identity element of $M$.
\end{example}

\section{Zero-divisors in a graded ring}\label{Sec:McCoy}

The following theorem gives an affirmative answer to Conjecture \ref{Conjecture I} and Conjecture \ref{Conjecture II} in the case that the group involved in the grading is a totally ordered group. But first we need to introduce a few useful tools. Let $R=\bigoplus\limits_{n\in G}R_{n}$ be a $G$-graded ring with $G$ a group. If $f= \sum\limits_{n\in G} r_{n}\in R$ with $r_{n}\in R_{n}$ for all $n\in G$, then we define the support of $f$ as $\Supp(f)= \{n\in G : r_{n} \neq 0\}$ which is a finite set. If $G$ is a totally ordered group and $f$ is a nonzero element,
then $n_{\ast}(f)$ denotes the smallest
element, and $n^{\ast}(f)$ denotes the largest element, of $\Supp(f)$ with respect to the ordering on $G$. We have $n_{\ast}(f)\leqslant n^{\ast}(f)$, and the equality holds if and only if $f$ is homogeneous. 
It is also clear that $n_{\ast}(f)+n_{\ast}(g)\leqslant n_{\ast}(fg)$ and $n^{\ast}(fg)\leqslant n^{\ast}(f)+n^{\ast}(g)$ for all $f,g\in R$ with $fg\neq0$.

The following theorem is one of the main technical results of this article. In this result, the group $G$ is not necessarily Abelian and the ring $R$ is not necessarily commutative: 

\begin{theorem}\label{th T-O} If $I$ is an unfaithful left ideal of a $G$-graded ring $R=\bigoplus\limits_{n\in G}R_{n}$ with $G$ a totally ordered group, then there exists a $($nonzero$)$ homogeneous element $g \in R$ such that $gI=0$. 
\end{theorem}

\begin{proof} Amongst all nonzero elements of $\Ann_{R}(I)$, by the well-ordering principle of the natural numbers, we may choose some $g$ such that the number $\ell:=|\Supp(g)|$ is minimal. 
We show that $\ell=1$, i.e. $g$ is homogeneous. We have $\ell\geqslant1$, because $g\neq0$.
Suppose that $\ell\geqslant 2$. Assume $g=\sum\limits_{n\in G}g_{n}$ with $g_{n}\in R_{n}$ for all $n$.
Put $s:=n^*(g)\in G$. Then $g_{s}\notin\Ann_{R}(I)$, because $|\Supp(g_s)|=|\{s\}|=1$. Thus there exists some $f=\sum\limits_{n\in G}f_{n}\in I$ (with $f_{n}\in R_{n}$ for all $n$) such that $g_{s}f\neq0$ and so $g_sf_k \neq 0$ for some $k\in G$.
It follows that $gf_k \neq 0$.
Choose $t\in\Supp(f)\subseteq G$ to be the largest element (with respect to the ordering) such that $gf_t \neq 0$. Thus $gf_{n}=0$ for all $n>t$. But $gf=0$ and so $(gf)_{s+t}=g_{s}f_t +\sum\limits_{\substack{x+y=s+t,\\(x,y)\neq(s,t)}}g_{x}f_{y}= 0$.  
Take $x,y\in G$ with $x+y=s+t$ and $(x,y)\neq(s,t)$. If $x<s$ and $y<t$ then, since $G$ is a totally ordered group, we will have $s+t=x+y<s+t$ which is impossible. Hence, $x>s$ or $y>t$.
If $x>s$ then $g_{x}=0$ and so $g_{x}f_{y}=0$. If $y>t$ 
then $gf_{y}=0$ and so $g_{x}f_{y}=0$.  
Thus $\sum\limits_{\substack{x+y=s+t,\\(x,y)\neq(s,t)}}g_{x}f_{y}=0$ and so $g_{s}f_{t}=0$.
Note that $\Ann_{R}(I)$ is a two-sided ideal of $R$, and so
$0\neq gf_t\in\Ann_{R}(I)$. We also have $|\Supp(gf_t)|\leqslant\ell-1$, because  
$g_{s} f_{t}=0$.
This contradicts the minimality of $\ell=|\Supp(g)|$.
Thus, $\ell=1$. Hence, $g$ is homogeneous.
\end{proof}

Note that in Theorem \ref{th T-O}, $I$ is an arbitrary (not necessarily graded) ideal. 

\begin{remark} Recall that if $R$ is a ring and $M$ is a monoid, then every element of the monoid-ring $R[M]=\bigoplus\limits_{m\in M}R$ can be written (uniquely) as $\sum\limits_{m\in M}r_{m}\epsilon_{m}$ with $r_{m}\in R$ where $\epsilon_{m}=(\delta_{m,n})_{n\in M}$ and $\delta_{m,n}$ is the Kronecker delta (note that if $r\in R$ then $r\epsilon_{m}$ denotes the element in $R[M]$ which has $r$ in the $m$ component and zero in all other components).
We have $\epsilon_{m}\cdot\epsilon_{n}=\epsilon_{m+n}$ for all $m,n\in M$ and $\epsilon_{0}$ is the unit (multiplication identity) of this ring where 0 is the identity element of $M$. Thus $R[M]$ is an $M$-graded ring with homogeneous components $R\epsilon_{m}=\{r\epsilon_{m}: r\in R\}$. If $M$ has the cancellation property, then $\epsilon_{m}$ is a non-zero-divisor element of $R[M]$ for all $m \in M$. 
In the literature, $r\epsilon_{m}$ is often denoted by $r\cdot m$, but we must add that our innovation of using $\epsilon_{m}$ instead of $m$ has many advantages. For instance, it considerably simplifies computations in monoid-rings. 
\end{remark}

\begin{example} Note that Theorem \ref{th T-O} fails if $G$ is not a totally ordered group. For example, $f=(1/2)(\epsilon_{0}+\epsilon_{1})$ is an idempotent element of the group-ring $R=\mathbb{Q}[\mathbb{Z}/2]=\mathbb{Q}\epsilon_{0}+
\mathbb{Q}\epsilon_{1}$ which is a $\mathbb{Z}/2$-graded ring. Hence, $f$ is a zero-divisor of $R$. In fact, $\Ann_{R}(f)=Rg$ where $g=(1/2)(\epsilon_{0}-\epsilon_{1})$.
But there is no homogeneous element in $R$ which could vanish $f$. 
\end{example}

\begin{corollary}\label{Coro 3-uc-Kap} Let $R$ be an $M$-graded commutative ring with $M$ a totally ordered cancellative monoid. If $f$ is a zero-divisor element of $R$, then there exists a (nonzero) homogeneous $g \in R$ such that $gf=0$.
\end{corollary}

\begin{proof} Since $R$ is a commutative ring, $I=Rf$ is an unfaithful ideal of $R$. Now the assertion follows from Theorem \ref{th T-O} by passing to the $G$-grading where $G$ is the Grothendieck group of $M$ which is a totally ordered group. 
\end{proof}

\begin{corollary}\label{cor:HomComZDisZD} If $M$ is a totally ordered cancellative monoid, then every homogeneous component of a zero-divisor element of an $M$-graded ring is a zero-divisor.
\end{corollary}

\begin{proof} It follows from Corollary \ref{Coro 3-uc-Kap}.
\end{proof}

\begin{example}
The converse of Corollary~\ref{cor:HomComZDisZD} does not hold.
To see this, consider the $\mathbb{N}$-graded ring $(\mathbb{Z}/6)[x]$ in which $2$ and $3x$ are zero-divisors. However, by Corollary~\ref{Corollary III}, the element $2+3x$ is not a zero-divisor.
\end{example}

In the following results (\ref{Coro 6-altin}, \ref{Coro 5-bes-Kap} and \ref{Coro TM 18}), the monoid $M$ is commutative, but the ring $R$ is not necessarily commutative:

\begin{corollary}\label{Coro 6-altin} If $I$ is an unfaithful left ideal of an $M$-graded ring $R$ with $M$ a totally ordered cancellative monoid, then there exists a $($nonzero$)$ homogeneous $g\in R$ such that $gI=0$.
\end{corollary}

\begin{proof} It follows from Theorem \ref{th T-O} by passing to the $G$-grading where $G$ is the Grothendieck group of $M$ which is a totally ordered group. 
\end{proof}

\begin{corollary}\label{Coro 5-bes-Kap} Let $R$ be a ring and $M$ a totally ordered cancellative monoid. If $I$ is an unfaithful left ideal of the monoid-ring $R[M]$, then $aI=0$ for some nonzero $a\in R$.
\end{corollary}

\begin{proof} We know that the monoid-ring $R[M]$ is an $M$-graded ring with homogeneous components $R\epsilon_{m}$. Then by Corollary \ref{Coro 6-altin}, there exists a (nonzero) homogeneous element $a\epsilon_{m}\in R[M]$  such that $\epsilon_{m}(aI)=(a\epsilon_{m})I=0$ where $0\neq a\in R$ and $m\in M$. But $\epsilon_{m}$ is a non-zero-divisor of $R[M]$, because $M$ has the cancellation property. Hence, $aI=0$.  
\end{proof}

Recall that the monoid-ring $R[\mathbb{N}]$ is called the ring of polynomials over $R$ with the variable $x:=\epsilon_{1}$ and is denoted by $R[x]$ where $\mathbb{N}=\{0,1,2,\ldots\}$ is the additive monoid of natural numbers. Then it is clear that $x^{n}=\epsilon_{n}$ for all $n\geqslant0$.
Similarly, the monoid-ring $R[\mathbb{N}\oplus\mathbb{N}]$ is called the ring of polynomials over $R$ with the variables $x:=\epsilon_{(1,0)}$ and $y:=\epsilon_{(0,1)}$ and is denoted by $R[x,y]$. Then we have $x^{m}y^{n}=\epsilon_{(m,n)}$ for all $m,n\geqslant0$.
More generally, the monoid-ring $R[\bigoplus\limits_{k\in S}\mathbb{N}]$ is called the ring of polynomials over $R$ with the variables $x_{k}:=\epsilon_{s_{k}}$ and is denoted by $R[x_{k}: k\in S]$ where $s_{k}=(\delta_{i,k})_{i\in I}\in \bigoplus\limits_{k\in S}\mathbb{N}$ and $\delta_{i,k}$ is the Kronecker delta.
 
Now the above result makes the multi-variable versions of McCoy's theorem quite easy, with no induction on the number of variables required:

\begin{corollary}\label{Coro TM 18} If $I$ is an unfaithful left ideal of the polynomial ring $R[x_{k}: k\in S]$, then $aI=0$ for some nonzero $a\in R$.
\end{corollary}

\begin{proof} It follows from Corollary \ref{Coro 5-bes-Kap}.
\end{proof}

\begin{corollary}\label{Corollary III} If $R$ is a commutative ring and $f$ is a zero-divisor element of the polynomial ring $R[x_{k}: k\in S]$, then $af=0$ for some nonzero $a\in R$.
\end{corollary}

\begin{proof} It follows from the above corollary.
\end{proof}

If $f$ is an element of a graded ring $R$, then $\mathrm{C}(f)$ denotes the ideal of $R$ generated by the homogeneous components of $f$. It is the smallest graded ideal containing $Rf$.  

\begin{corollary}\label{coro 4 zero divisor}
If $R=\bigoplus\limits_{m\in M}R_{m}$ is an $M$-graded commutative ring with $M$ a totally ordered cancellative monoid, then $Z(R)\subseteq\bigcup
\limits_{\mathfrak{p}\in\Spec(R)}\mathfrak{p}^{\ast}$.
\end{corollary}

\begin{proof}
If $f=\sum\limits_{m\in M}r_{m}\in Z(R)$ with $r_{m}\in R_{m}$ for all $m$, then by Corollary \ref{Coro 3-uc-Kap}, there exists a (nonzero) homogeneous $g\in R$ such that $fg=0$ and so $r_{m}g=0$ for all $m\in G$. Since $g\neq0$, we get that
$\mathrm{C}(f)$ is a proper ideal of $R$. Thus, there exists a prime ideal $\mathfrak{p}$ of $R$ such that $f\in\mathrm{C}(f)\subseteq\mathfrak{p}$. But $\mathrm{C}(f)\subseteq\mathfrak{p}^{\ast}$, because $\mathrm{C}(f)$ is a graded ideal.
\end{proof}

\begin{remark}\label{Remark 2 iki}
Note that Theorem~\ref{th T-O} cannot be generalized to modules. More precisely, the statement ``if $M$ is an unfaithful module over a $G$-graded ring $R$ with $G$ a totally ordered group, then there exists a nonzero homogeneous $g\in R$ such that $gM=0$" is false.
As a counterexample, $M=(\mathbb{Z}/2)[x]/(1+x)$ is an unfaithful module over the $\mathbb{N}$-graded ring $R=(\mathbb{Z}/2)[x]$. Suppose that there is a nonzero homogeneous $g\in R$ such that $gM=0$. Then $g\in\Ann_{R}(M)=(1+x)$. Clearly, $g=x^{d}$ for some $d\geqslant1$. Hence $x^{d}=(1+x)f(x)$. But $R$ is a UFD and the elements $x$ and $1+x$ are irreducible. 
This yields that $x=1+x$ and so $1=0$ which is a contradiction.
\end{remark}

\section{Units in a graded ring}\label{Sec:Units}

In this section, especially in Theorem \ref{Theorem ix} (and also in Theorem \ref{coro ET new 20}), we give a complete characterization of invertible elements in $G$-graded commutative rings with $G$ is a torsion-free Abelian group.  

\begin{lemma}\label{Lemma I}
Let $f$ be an invertible homogeneous element of an $M$-graded ring $R=\bigoplus\limits_{m\in M}R_{m}$ with $M$  a cancellative monoid. Then $f^{-1}$ is homogeneous with $\dg(f)+\dg(f^{-1})=0$.
\end{lemma}

\begin{proof} Suppose $f^{-1}=\sum\limits_{m\in M}r_{m}$ where $r_{m}\in R_{m}$ for all $m$. We have $\sum\limits_{m\in M}fr_{m}=1$. Since $M$ has the cancellation property, so on the left hand side of the above equality, $fr_{m}$ is the only homogeneous element of degree $\dg(f)+m$ for all $m\in \Supp(f^{-1})$. Thus there exists a (unique) $d\in M$ such that $fr_{d}=1$ and  $fr_{m}=0$ and so $r_{m}=0$ for all $m\neq d$. This shows that $f^{-1}=r_{d}$ and $\dg(f)+\dg(f^{-1})=0$. 
\end{proof}

\begin{lemma}\label{Theorem V}
Every invertible element of a $G$-graded integral domain with $G$ a torsion-free Abelian group is homogeneous. 
\end{lemma}

\begin{proof} Let $R$ be a $G$-graded integral domain. 
Then we have $n_{\ast}(fg)=n_{\ast}(f)+n_{\ast}(g)$ and $n^{\ast}(fg)=n^{\ast}(f)+n^{\ast}(g)$ for all nonzero $f,g\in R$. If $fg=1$, then $n_{\ast}(fg)=n^{\ast}(fg)=0$. If $n_{\ast}(f)< n^{\ast}(f)$, then $0=n_{\ast}(f)+n_{\ast}(g)< n^{\ast}(f)+n_{\ast}(g)\leqslant n^{\ast}(f)+n^{\ast}(g)=0$  which is impossible. Therefore, $n_{\ast}(f)=n^{\ast}(f)$, showing that $f$ is homogeneous.
\end{proof}

\begin{remark}\label{remark i nice example}
Note that Lemma \ref{Theorem V} cannot be generalized to $G$-graded rings which are only reduced or whose base subrings $R_{0}$ are integral domains.
For the first case, consider the reduced $\mathbb{Z}$-graded ring $(\mathbb{Z}/6)[x,x^{-1}]$.  The element $g=2x+3x^{-1}$ is clearly not homogeneous, but it is invertible with the inverse
$g^{-1}=2x^{-1}+3x$.
For the second case, consider the associated graded ring $\gr_{\mathfrak{p}}(R)=
\bigoplus\limits_{n\geqslant0}\mathfrak{p}^{n}/\mathfrak{p}^{n+1}=
R/\mathfrak{p}\oplus
\mathfrak{p}/\mathfrak{p}^{2}\oplus\cdots$ where $\mathfrak{p}$ is a prime ideal of a ring $R$ with the property that there is some $f\in\mathfrak{p}\setminus\mathfrak{p}^{2}$ such that $f^{2}\in\mathfrak{p}^{3}$. In this case, $f+\mathfrak{p}^{2}$ is nilpotent, since $(f+\mathfrak{p}^{2})^{2}=f^{2}+\mathfrak{p}^{3}=0$. Thus, the element $(f^{n}+\mathfrak{p}^{n+1})_{n\geqslant0}=
(1+\mathfrak{p},f+\mathfrak{p}^{2},0,0,0,\ldots)$ is not homogeneous, but invertible in $\gr_{\mathfrak{p}}(R)$, because the sum of an invertible element and a nilpotent element is invertible. Finding such a prime ideal is not hard. For instance, in the ring $\mathbb{Z}/4$ we may take $\mathfrak{p}=\{0,2\}$ and $f=2$.
\end{remark} 

\begin{corollary}\label{Coro 17 onyedi} If $M$ is a totally ordered cancellative monoid, then every invertible element of an $M$-graded integral domain is homogeneous.
\end{corollary}

\begin{proof} It follows from Lemma \ref{Theorem V} by passing to the $G$-grading where $G$ is the Grothendieck group of $M$. 
\end{proof}

By $U(R)$ we denote the group of units (invertible elements) of a ring $R$. 

\begin{remark} Let $R$ be an integral domain and $M$ a totally ordered monoid with the cancellation property. Then by Corollary \ref{Coro onalti 16}, $R[M]$ is an integral domain. If $f$ is an invertible element of $R[M]$ then by Corollary \ref{Coro 17 onyedi}, $f=a\epsilon_{x}$ (and by Lemma \ref{Lemma I}, $f^{-1}=a^{-1}\epsilon_{y}$) where $a\in U(R)$ and $x+y=0$ for some $x,y\in M$. But note that if $r\in U(R)$ and $m\in M$ then $r\epsilon_{m}$ is not necessarily invertible in $R[M]$. In fact, $r\epsilon_{m}$ is invertible in $R[M]$ if and only if $r\in U(R)$ and $m$ is invertible in $M$. This observation leads us to the following result.
\end{remark}

\begin{corollary}\label{Coro 2-iki-Kap} If $R$ is an integral domain and $G$ is a torsion-free Abelian group, then the units of $R[G]$ are precisely of the form $r\epsilon_{x}$ with $r\in U(R)$ and $x\in G$.
\end{corollary}

\begin{proof} We know that $R[G]$ is a $G$-graded ring with the homogeneous components $R\epsilon_{x}$. Then by Lemma \ref{Lemma 2-iki}, the zero ideal of  $R[G]$ is a prime ideal. 
Hence, $R[G]$ is a $G$-graded integral domain. Now the assertion easily follows from Lemmas  \ref{Lemma I} and \ref{Theorem V}. 
\end{proof}

If $G$ is not a torsion-free Abelian group, then the group-ring $K[G]$ is not an integral domain even if $K$ is a field. For example, the group-ring $\mathbb{Q}[\mathbb{Z}/2]$ has nontrivial idempotents $(1/2)(\epsilon_{0}+\epsilon_{1})$ and $(1/2)(\epsilon_{0}-\epsilon_{1})$.

\begin{lemma}\label{Lemma 3-easy}
Let $f=\sum\limits_{k=1}^{n}r_{k}$ be an element of a ring $R$. If $(r_{1},\ldots,r_{n})=R$ and $r_{i}r_{k}$ is nilpotent for all $i\neq k$, then $f$ is invertible in $R$. 
\end{lemma}

\begin{proof} It suffices to show that $Rf=R$. If $Rf\neq R$, then $Rf\subseteq\mathfrak{p}$ for some prime ideal $\mathfrak{p}$ of $R$. But there exists some $k$ such that $r_{k}\notin\mathfrak{p}$. By assumption, $r_{i}r_{k}\in\mathfrak{p}$ and hence $r_{i}\in\mathfrak{p}$ for all $i\neq k$. Thus, $r_{k}=f-\sum\limits_{i\neq k}r_{i}\in\mathfrak{p}$ which is a contradiction.
\end{proof}

\begin{theorem}\label{Theorem ix}
Let $f=\sum\limits_{k\in G}r_{k}$ be an element of a $G$-graded ring $R=\bigoplus\limits_{k\in G}R_{k}$ with $G$ a torsion-free Abelian group and $r_{k}\in R_{k}$ for all $k$. Then $f$ is invertible in $R$ if and only if $(r_{k}: k\in G)=R$ and $r_{i}r_{k}$ is nilpotent for all $i\neq k$.
\end{theorem}

\begin{proof}
If $f$ is invertible in $R$ then 
$R=Rf\subseteq(r_{k}: k\in G)\subseteq R$ and so $(r_{k}: k\in G)=R$. To prove that $r_{i}r_{k}$ is nilpotent for $i\neq k$, it suffices to show that $r_{i}r_{k}\in\mathfrak{p}$
for every minimal prime ideal $\mathfrak{p}$ of $R$. 
Since $G$ is a totally ordered group, thus by Corollary \ref{Coro 14-ondort}, every minimal prime ideal of $R$ is a graded ideal. Hence, $R/\mathfrak{p}$ is a $G$-graded integral domain. Clearly $f+\mathfrak{p}$ is invertible in $R/\mathfrak{p}$.
Thus, by Lemma~\ref{Theorem V}, there exists some $\ell\in G$ such that $r_{n}\in\mathfrak{p}$ for all $n\neq\ell$.
Then $r_{i}$ or $r_{k}$ is always a member of $\mathfrak{p}$ and so $r_{i}r_{k}\in\mathfrak{p}$ for all $i\neq k$. The reverse implication follows from Lemma \ref{Lemma 3-easy}.
\end{proof}

\begin{corollary}\label{Coro 12-oniki}
Let $f=\sum\limits_{x\in M}r_{x}$ be an invertible element of an $M$-graded ring $R=\bigoplus\limits_{x\in M}R_{x}$ with $M$ a totally ordered cancellative monoid and $r_{x}\in R_{x}$ for all $x$. If $r_{0}$ is invertible in $R$, then $r_{x}$ is nilpotent for all $x\neq0$.
\end{corollary}

\begin{proof} Clearly $f$ is invertible in the $G$-graded ring $R$ where $G$ is the Grothendieck group of $M$ which is a totally ordered group. Thus by Theorem \ref{Theorem ix}, $r_{0}r_{x}$ and so $r_{x}$ are nilpotent for all $x\neq0$. 
\end{proof}

\begin{example} The above results (Lemma \ref{Theorem V}, Theorem \ref{Theorem ix} and Corollary \ref{Coro 12-oniki}) fail to hold if $G$ is not a torsion-free Abelian group. For example, $f=(1/3)(\epsilon_{0}+\epsilon_{1}+\epsilon_{2})$ is an idempotent element of the group-ring $R=\mathbb{Q}[\mathbb{Z}/3]=\mathbb{Q}\epsilon_{0}+
\mathbb{Q}\epsilon_{1}+\mathbb{Q}\epsilon_{2}$ which is a $\mathbb{Z}/3$-graded ring. Hence, $1-2f=(1/3)(\epsilon_{0}-2\epsilon_{1}-2\epsilon_{2})$ is invertible in $R$ which is neither homogeneous nor its distinct double products are nilpotent (recall that if $e$ is an idempotent of a ring then $1-2e$ is invertible, because $(1-2e)^{2}=1$).  
\end{example}

Recall that in any ring, the sum of an invertible element and a nilpotent element is an invertible element. 

\begin{corollary}\label{Coro 13-onuc}
Let $f=\sum\limits_{x\in M}r_{x}$ be an element of an $M$-graded ring $R=\bigoplus\limits_{x\in M}R_{x}$ with $M$ a totally ordered cancellative monoid and $r_{x}\in R_{x}$ for all $x$. If the identity element of $M$ is the smallest member, then $f$ is invertible in $R$ if and only if  $r_{0}$ is invertible in $R_{0}$ and $r_{x}$ is nilpotent for all $x\neq0$.
\end{corollary}

\begin{proof} If $f$ is invertible in $R$ then $fg=1$ for some $g=\sum\limits_{x\in M}r'_{x}\in R$. It follows that $r_{0}r'_{0}+\sum
\limits_{\substack{x+y=0,\\(x,y)\neq(0,0)}}r_{x}r'_{y}=1$. But if $(x,y)\neq(0,0)$ for some $x,y\in M$ then $0<x+y$, because the identity element of $M$ is the smallest member. This yields that $\sum
\limits_{\substack{x+y=0,\\(x,y)\neq(0,0)}}r_{x}r'_{y}=0$ and so $r_{0}r'_{0}=1$. Then by Corollary \ref{Coro 12-oniki}, $r_{x}$ is nilpotent for all $x\neq0$. The reverse implication is clear.   
\end{proof}

\begin{corollary}\label{Theorem I}
Let $f=\sum\limits_{n\geqslant0}r_{n}$ be an element of an $\mathbb{N}$-graded ring $R=\bigoplus\limits_{n\geqslant0}R_{n}$ with $r_n \in R_n$ for all $n$. Then $f$ is invertible in $R$ if and only if $r_{0}$ is invertible in $R_{0}$ and $r_{n}$ is nilpotent for all $n\geqslant1$.
\end{corollary}

\begin{proof}
It follows from Corollary \ref{Coro 13-onuc}.
\end{proof}

The following two results are immediate consequences of Corollary~\ref{Theorem I}.

\begin{corollary}\label{Coro in contrast with Z}
In an $\mathbb{N}$-graded ring, every invertible homogeneous element is of degree zero.
\end{corollary}

\begin{corollary}\label{Corollary ii}
In a reduced $\mathbb{N}$-graded ring, every invertible element is homogeneous of degree zero.
\end{corollary}

\begin{corollary} An element of the polynomial ring $R[x_{k}: k\in I]$ is invertible if and only if its constant term in invertible in $R$ and all the remaining coefficients are nilpotent. 
\end{corollary}

\begin{proof} The identity element of the additive monoid $M=\bigoplus\limits_{k\in I}\mathbb{N}$ is the smallest element with respect to the lexicographical ordering. We know that $R[x_{k}: k\in I]=R[M]$. Hence the assertion follows from Corollary \ref{Coro 13-onuc} with taking into account that each variable $x_{k}$ is a non-zero-divisor. 
\end{proof}

\begin{corollary}\label{coro 1 old}
Let $R$ be an $\mathbb{N}$-graded ring. Then $R$ is reduced if and only if $R_{0}$ is reduced and $U(R_{0})=U(R)$.
\end{corollary}

\begin{proof}
The ``only if'' statement follows from Corollary~\ref{Theorem I}.
Conversely, it suffices to show that every nilpotent homogeneous element $f\in R$ of positive degree is zero. Clearly, $1+f$ is invertible in $R$. Thus, by assumption, $f=0$.
\end{proof}

\section{Idempotents in a graded ring}

In this section, we give a complete characterization of idempotent elements in a $G$-graded commutative ring with $G$ is an Abelian group. 

It is well-known that every idempotent element of an $\mathbb{N}$-graded ring $R=\bigoplus\limits_{n\geqslant0}R_{n}$ is contained in the base subring $R_{0}$. In the following theorem we will show that the same assertion holds for $\mathbb{Z}$-graded rings and more generally for $G$-graded rings with $G$ a torsion-free Abelian group. But note that the $\mathbb{N}$-graded case argument no longer works here (the proof is a bit more subtle). Our theorem also gives a natural and quite simple proof for Kirby's theorem \cite[Theorem 1]{Kirby}, which was already proved by very difficult methods and as if with the additional ``strongly graded'' (i.e. $R_{m}R_{n}=R_{m+n}$) assumption.

\begin{theorem}\label{Lemma T-K} Every idempotent element of a $G$-graded ring  $R=\bigoplus\limits_{n\in G}R_{n}$ with $G$ a torsion-free Abelian group is contained in the base subring $R_{0}$.
\end{theorem}

\begin{proof} Let $f=\sum\limits_{n\in G}r_{n}$ be an idempotent of $R$ with $r_{n}\in R_{n}$ for all $n$. If $\mathfrak{p}$ is a minimal prime ideal of $R$, then $f+\mathfrak{p}$ is an idempotent of the integral domain $R/\mathfrak{p}$. Thus $f\in\mathfrak{p}$ or $1-f\in\mathfrak{p}$.
Since $G$ is a totally ordered group, 
thus by Corollary \ref{Coro 14-ondort},
every minimal prime ideal of $R$ is a graded ideal. Hence, $r_{n}\in\mathfrak{p}$ and so $r_{n}$ is nilpotent for all $n\neq0$. 
It follows that $f+\mathfrak{N}=r_{0}+\mathfrak{N}$ where $\mathfrak{N}$ is the nilradical of $R$. This shows that $r_{0}+\mathfrak{N}_{0}$ is an idempotent of the ring $R_{0}/\mathfrak{N}_{0}$ where $\mathfrak{N}_{0}=\mathfrak{N}\cap R_{0}$ is the nilradical of $R_{0}$.
But it is well known that the idempotents of any ring can be lifted modulo its nilradical (see e.g. \cite{Tarizadeh-Sharma}). So there exists an idempotent $e\in R_{0}$ such that $a:=r_{0}-e$ is nilpotent. It follows that $(1-e)r_{0}=(1-e)a$ and $e(1-r_{0})=-ea$ are nilpotent. Also $(1-e)r_{n}$ and $er_{n}$ are nilpotent for all $n\neq0$.
This shows that $(1-e)f$ and $e(1-f)$ are nilpotent elements of $R$.  
But it can be seen that if $e,e'$ are idempotents of a ring such that $e-e'$ is a member of the Jacobson radical, then $e=e'$ (see \cite[Lemma 3.8]{Tarizadeh-Sharma}). Hence, $f=ef=e\in R_{0}$.
\end{proof}

The following corollaries are immediate consequences of Theorem \ref{Lemma T-K}. 

\begin{corollary} Every idempotent element of a $\mathbb{Z}$-graded ring $R=\bigoplus\limits_{n\in\mathbb{Z}}R_{n}$ is contained in the base subring $R_{0}$.
\end{corollary}

For any ring $R$ and a monoid $M$, it can be seen that the localization of $R[M]$ with respect to the multiplicative set $S=\{\epsilon_{m}: m\in M\}$ is canonically isomorphic to the group-ring $R[G]$ as $G$-graded rings where $G$ is the Grothendieck group of $M$. In particular, the localization of the polynomial ring $R[x]$ with respect to the multiplicative set $S=\{1,x,x^{2},\cdots\}$ is denoted by $R[x,x^{-1}]$ (it is called the ring of Laurent polynomials over $R$) and we have the canonical isomorphism of $\mathbb{Z}$-graded rings $R[\mathbb{Z}]\simeq R[x,x^{-1}]$. 

\begin{corollary} For any ring $R$, every idempotent element of the ring of Laurent polynomials $R[x,x^{-1}]$ is contained in $R$.
\end{corollary}

Recall from  \cite[\S4]{Tarizadeh-Taheri} that an ideal of a ring is called a \emph{regular ideal} if it is generated by a set of idempotent elements.

\begin{corollary} If $G$ is a torsion-free Abelian group, then the following assertions hold: \\
$\mathbf{(i)}$ Every regular ideal of a $G$-graded ring is a graded ideal. \\
$\mathbf{(ii)}$ Let $R$ be a $G$-graded ring. Then $\Spec(R)$ is connected if and only if $\Spec(R_{0})$ is connected.
\end{corollary} 

In addition to the above important consequences, Theorem \ref{Lemma T-K} also leads us to derive the following general case:

\begin{theorem}\label{Theorem i-T} Every idempotent of a $G$-graded ring $R=\bigoplus\limits_{n\in G}R_{n}$ with $G$ an Abelian group is contained in the subring $\bigoplus\limits_{h\in H}R_{h}$ where $H$ is the torsion subgroup of $G$. 
\end{theorem}

\begin{proof} Let $f$ be an idempotent of $R$. Let $K$ be the subgroup of $G$ generated by the finite subset $\Supp(f)$. Then $f$ is an idempotent element of the subring $\bigoplus\limits_{x\in K}R_{x}$, because  $\Supp(f)\subseteq K$. Hence, without loss of generality, we may assume that $G$ is a finitely generated group. Then by the fundamental theorem of finitely generated Abelian groups, we may write $G=H\oplus F$ where $F$ is a (finitely generated) torsion-free subgroup of $G$. Then by changing the grading, $R=\bigoplus\limits_{x\in F}S_{x}$ can be viewed as an $F$-graded ring with the homogeneous components $S_{x}=\sum\limits_{\substack{n\in G,\\n-x\in H}}R_{n}$. Then by Theorem \ref{Lemma T-K}, we have $f\in S_{0}=\sum\limits_{n\in H}R_{n}$. 
\end{proof} 

\begin{lemma}\label{Lemma 1 bir-one} Let  $R=\bigoplus\limits_{m\in M}R_{m}$ be an $M$-graded ring with $M$ a monoid. Then $\bigoplus\limits_{m\in M^{0}}R_{m}$ is a subring of $R$ where we call $M^{0}=\{x\in M:\exists y\in M, x+y=y\}$ the quasi-zero submonoid of $M$.
\end{lemma}

\begin{proof} Let $G=\{[a,b]: a,b\in M\}$ be the Grothendieck group of $M$ and let $M'=\{[m,0]: m\in M\}$ be the image of the canonical map $f:M\rightarrow G$. Thus $M'$ is a submonoid of the group $G$, and hence $M'$ is cancellative. Then by changing the grading, we can view $R=\bigoplus\limits_{x\in M'}S_{x}$ as an $M'$-graded ring with the homogeneous components $S_{x}=\sum\limits_{[m,0]=x}R_{m}$. But $M^{0}=\Ker(f)=f^{-1}(0)$. Thus by Lemma \ref{Lemma 5-five-bes}, $S_{0}=\bigoplus\limits_{m\in M^{0}}R_{m}$ is a subring of $R$. In particular, $1\in\bigoplus\limits_{m\in M^{0}}R_{m}$. 
\end{proof}

If $M$ is a commutative monoid, then it can be easily seen that the set $\widetilde{M}=\{x\in M:\exists n\geqslant1, \exists y\in M, nx+y=y\}$ is a submonoid of $M$. We call $\widetilde{M}$ the quasi-torsion submonoid of $M$. Note that the torsion submonoid $T(M)=\{x\in M:\exists n\geqslant1, nx=0\}$ is contained in $\widetilde{M}$ and equality holds if $M$ has the cancellation property. 
 
\begin{theorem}\label{Theorem 3-III} Every idempotent of an $M$-graded ring $R=\bigoplus\limits_{m\in M}R_{m}$ with $M$ a monoid is contained in the subring $\bigoplus\limits_{m \in \widetilde{M}}R_{m}$ where $\widetilde{M}$ is the quasi-torsion submonoid of $M$.
\end{theorem}

\begin{proof} First note that by Lemma \ref{Lemma 1 bir-one}, if $L$ is a submonoid of $M$ containing $M^{0}$, then  $\bigoplus\limits_{m\in L}R_{m}$ is a subring of an $M$-graded ring $R=\bigoplus\limits_{m\in M}R_{m}$. In particular, since $M^{0}\subseteq M^{\ast}$, thus $\bigoplus\limits_{m \in M^{\ast}}R_{m}$ is a subring of $R$. Then let $G=\{[a,b]: a,b\in M\}$ be the Grothendieck group of $M$. Consider the canonical morphism of monoids $f:M\rightarrow G$ which is given by $m\mapsto[m,0]$. By changing the grading,   
we can make $R=\bigoplus\limits_{x\in G}S_{x}$ into a $G$-graded with the homogeneous components $S_{x}=\sum\limits_{[m,0]=x}R_{m}$ (note that $S_{x}=0$ for all $x\in G\setminus\Ima f$). Then by Theorem \ref{Theorem i-T}, every idempotent of $R$ is contained in the subring $\bigoplus\limits_{x\in H}S_{x}$ where $H$ is the torsion subgroup of $G$. But $M^{\ast}=f^{-1}(H)$. Hence,  $\bigoplus\limits_{x\in H}S_{x}=\bigoplus\limits_{m \in M^{\ast}}R_{m}$.
\end{proof}

The following result generalizes Bass' theorem (see \cite[Lemma 6.7]{Bass} or \cite[Corollary 4]{Kanwar-Leroy}). 

\begin{corollary}\label{Coro 6-six} For any ring $R$ and a monoid $M$, then every idempotent of the monoid-ring $R[M]$ is contained in the subring $R[\widetilde{M}]$ where $\widetilde{M}$ is the quasi-torsion submonoid of $M$. 
\end{corollary}

\begin{proof} We know that the monoid-ring $R[M]=\bigoplus\limits_{m\in M}S_{m}$ is an $M$-graded ring with homogeneous components $S_{m}=R\epsilon_{m}$. Then by Theorem \ref{Theorem 3-III}, every idempotent of $R[M]$ is contained in 
the subring $\bigoplus\limits_{m\in \widetilde{M}}S_{m}=R[\widetilde{M}]$.
\end{proof}

\begin{corollary}\label{Lemma B-T} If $M$ is a torsion-free monoid with the cancellation property, then every idempotent of the monoid-ring $R[M]$ is contained in $R$. 
\end{corollary}

\begin{proof} By hypothesis, $\widetilde{M}=0$ and so $R[\widetilde{M}]=R$. Hence, the assertion follows from the above corollary.
\end{proof}

\section{Homogeneity of the Jacobson radical and colon ideal}

In the following result we generalize Bergman's theorem to the more general setting of $G$-graded commutative rings, and give a natural and quite elementary proof of it.

\begin{theorem}\label{Theorem II}
The Jacobson radical of every $G$-graded ring with $G$ a torsion-free Abelian group is a graded ideal.
\end{theorem}

\begin{proof}
Let $R=\bigoplus\limits_{n\in G}R_{n}$ be a $G$-graded ring. It suffices to show that $\mathfrak{J}(R)=\bigcap
\limits_{\mathfrak{m}\in\Max(R)}\mathfrak{m}^{\ast}$. The inclusion $\bigcap\limits_{\mathfrak{m}\in\Max(R)}\mathfrak{m}^{\ast}
\subseteq\mathfrak{J}(R)$ is clear. To prove the reverse inclusion, take $f=\sum\limits_{n\in G}r_{n}\in\mathfrak{J}(R)$ with $r_{n}\in R_{n}$ for all $n$. Suppose there is a maximal ideal $\mathfrak{m}$ of $R$ such that $f\notin\mathfrak{m}^{\ast}$. Then there exists some $0\neq d\in G$ such that $r_{d}\notin\mathfrak{m}^{\ast}$, because if $r_{n}\in\mathfrak{m}^{\ast}$ for all $n\neq0$, then $r_{0}=f-\sum\limits_{n\neq0}r_{n}\in\mathfrak{m}$ and hence $r_{0}\in\mathfrak{m^{\ast}}$, yielding $f\in\mathfrak{m^{\ast}}$ which is a contradiction. Since $G$ is a totally ordered group, thus by Corollary \ref{Coro 20 twenty yirmi}, the graded ideal $\mathfrak{m}^{\ast}$ is a prime ideal of $R$. We know that $1+r_{d}f$ is invertible in $R$ and thus its image is invertible in the $G$-graded integral domain $R/\mathfrak{m}^{\ast}$. By Lemma \ref{Theorem V}, in a $G$-graded integral domain, every invertible element is homogeneous. Thus there exists some $k$ such that the homogeneous component $(1+r_{d}f)_{n}\in\mathfrak{m}^{\ast}$ for all $n\neq k$. We have $(1+r_{d}f)_{n}=r_{d}r_{n-d}+\delta_{0,n}$ where $\delta_{0,n}$ is the Kronecker delta. It follows that $k=2d$. If $n\neq d, -d$ then $n+d\neq 2d,0$ and thus $r_{d}r_{n}=(1+r_{d}f)_{n+d}\in\mathfrak{m}^{\ast}$. Therefore, $r_{n}\in\mathfrak{m}^{\ast}$ for all $n\neq d, -d$.
This yields $r_{d}+r_{-d}-f\in\mathfrak{m}^{\ast}$. Hence, $1+r_{d}+r_{-d}+\mathfrak{m}^{\ast}$ is invertible in $R/\mathfrak{m}^{\ast}$. Again by Lemma \ref{Theorem V}, we get that $r_{d}\in\mathfrak{m}^{\ast}$ which is a contradiction. This completes the proof. 
\end{proof}

\begin{corollary} If $M$ is a totally ordered monoid with the cancellation property, then the Jacobson radical of an $M$-graded ring is a graded ideal.
\end{corollary}

\begin{proof} It follows from Theorem \ref{Theorem II} by passing to the $G$-grading where $G$ is the Grothendieck group of $M$ which is a torsion-free Abelian group. 
\end{proof}

Recall that if $I$ and $J$ are ideals of a ring $R$, then the ideal quotient (or, colon ideal) of $I$ by $J$ is the ideal $I:_{R}J=\{r\in R: rJ\subseteq I\}$ of $R$.

Let $(X,<)$ be a well-ordered set and let $P$ be a property of elements of $X$, i.e. $P(x)$ is a mathematical statement for all $x\in X$. Suppose that whenever $P(y)$ is true for all $y<x$, then $P(x)$ is also true. Then it is easy to see that $P(x)$ is true for all $x\in X$. This statement is called the transfinite induction. Every finite totally ordered set is well-ordered, and hence the transfinite induction can be applied for such sets. Using this weak version of the transfinite induction, we obtain the following general result: 

\begin{theorem}\label{Theorem vii}
Let $I$ be a graded radical ideal of a $G$-graded ring $R=\bigoplus\limits_{i\in G}R_{i}$ with $G$ a torsion-free Abelian group and $J$ be an arbitrary ideal of $R$. Then $I:_{R}J$ is a graded ideal.
\end{theorem}

\begin{proof} Take $f=\sum\limits_{i\in G}r_{i} \in I:_{R}J$ where $r_{i}\in R_{i}$ for all $i$. We must prove that each $r_{i}\in I:_{R}J$. Take $g=\sum\limits_{k\in G}r'_{k}\in J$ where $r'_{k}\in R_{k}$ for all $k$. It suffices to show that $r_{i}r'_{k}\in I$ for all $i,k\in G$. Since $I$ is a radical ideal, it will be enough to show that for each pair $(i,k)\in X=\Supp(f)\times\Supp(g)$ then $r_{i}r'_{k}\in\mathfrak{p}$ where $\mathfrak{p}$ is a prime ideal of $R$ containing $I$. We will prove the assertion by induction on the finite well-ordered set $S=\{i+k: (i,k)\in X\}=\{d+s,\ldots, m+l\}$
where $f=r_{d}+\dots+r_{m}$ and $g=r'_{s}+\dots+r'_{l}$ with $d=n_{\ast}(f)\leqslant m=n^{\ast}(f)$ and $s=n_{\ast}(g)\leqslant l=n^{\ast}(g)$.
We have $(fg)_{d+s}=\sum\limits_{i+k=d+s}r_{i}r'_{k}$.  
If $i+k=d+s$ and $(i,k)\neq(d,s)$, then since $G$ is a totally ordered group, we have $i<d$ or $k<s$ and so $r_{i}r'_{k}=0$. Thus $(fg)_{d+s}=r_{d}r'_{s}$.
But $r_{d}r'_{s}\in I\subseteq\mathfrak{p}$, because $I$ is a graded ideal and $fg \in I$.
We have thus established the base case of the induction ($n=d+s$).
Assume now $n>d+s$ with $n\in S$. By the induction hypothesis, if $i+k<n$ for some $(i,k)\in X$ then $r_i r'_k \in \mathfrak{p}$.
Seeking a contradiction, suppose that there exists $(a,b)\in X$ with $a+b=n$ for which $r_{a}r'_{b}\notin\mathfrak{p}$.
Using that $I$ is graded, we have $(fg)_{n}=r_{a}r'_{b}+\sum\limits_{\substack{i+k=n,\\
(i,k)\neq(a,b)}}r_{i}r'_{k}\in I \subseteq \mathfrak{p}$.
If $i+k=n$ and $(i,k)\neq(a,b)$, then $i<a$
or $a<i$.
If $i<a$, then $i+b<n$, because $G$ is a totally ordered group. Hence by the induction, $r_{i}r'_{b}\in\mathfrak{p}$ which yields $r_{i} \in \mathfrak{p}$ (since $r'_{b}\notin\mathfrak{p}$).
Similarly, if $a<i$, then $a+k<n$
and hence by the induction,
$r_{a}r'_{k}\in\mathfrak{p}$ which yields $r'_{k} \in \mathfrak{p}$ (since $r_{a}\notin\mathfrak{p}$).
Therefore for any such pair $(i,k)$ we have $r_{i}r'_{k}\in\mathfrak{p}$  and so
$\sum\limits_{\substack{i+k=n,\\
(i,k)\neq(a,b)}}r_{i}r'_{k}\in \mathfrak{p}$. It follows that $r_{a}r'_{b}\in\mathfrak{p}$
which is a contradiction. This complete the proof. 
\end{proof}

\begin{corollary}\label{Corollary ix}
If $I$ is an ideal of a $G$-graded ring $R$ with $G$ a torsion-free Abelian group, then the colon ideal $\mathfrak{J}(R):_{R}I$ is a graded ideal.
\end{corollary}

\begin{proof}
By Theorem~\ref{Theorem II}, the Jacobson radical of $R$ is a graded ideal. It is also a radical ideal. Hence, the assertion follows from Theorem \ref{Theorem vii}.
\end{proof}

\begin{remark} Note that in the proof of Theorem \ref{Theorem vii} we cannot do the transfinite induction on the set $\{n\in G: d+s\leqslant n\leqslant m+l\}$, because this totally ordered set is not necessarily well-ordered (nor finite). For example, if $G=\mathbb{Z}^{2}$ then $(0,1)<(1,n)<(2,0)$ for all $n\in\mathbb{Z}$. Also note that although $I:_{R}J=\Ann_{R}\big((I+J)/I\big)$ as $R$-modules, we cannot deduce Theorem \ref{Theorem vii} from Theorem \ref{th T-O}, because Theorem \ref{th T-O} cannot be generalized to the setting of modules (see Remark \ref{Remark 2 iki}). 
\end{remark}

In addition to its generalization to the setting of graded rings, the other main novelty and power of the above theorem is that $J$ is an arbitrary (not necessarily graded) ideal. Some consequences of the above theorem are given below.

\begin{corollary} Let $I$ be a graded radical ideal and $J$ an arbitrary ideal of an $M$-graded ring $R$ with $M$ a totally ordered cancellative monoid. Then $I:_{R}J$ is a graded ideal.
\end{corollary}

\begin{proof} It follows from Theorem~\ref{Theorem vii} by passing to the $G$-grading where $G$ is the Grothendieck group of $M$ which is a totally ordered group. 
\end{proof}

\begin{corollary}\label{coro 3 old} If $I$ is an ideal of a $G$-graded ring $R$ with $G$ a torsion-free Abelian group, then $\mathfrak{N}:_{R}I$ is a graded ideal.
\end{corollary}

\begin{proof}
It follows from Theorem~\ref{Theorem vii},
because the
nilradical of $R$ is a graded radical ideal.
\end{proof}

\begin{corollary}\label{Coro xii}
If $I$ is an ideal of a reduced $G$-graded ring $R$ with $G$ a torsion-free Abelian group, then $\Ann_{R}(I)$ is a graded ideal.
\end{corollary}

\begin{proof}
By Corollary \ref{coro 3 old}, $\mathfrak{N}(R):_{R}I=0:_{R}I=\Ann_{R}(I)$ is a graded ideal.
\end{proof}

\begin{corollary}\label{Theorem IV}
Let $I$ be a graded radical ideal of a $G$-graded ring $R=\bigoplus\limits_{i\in G}R_{i}$ and let $f=\sum\limits_{i\in G}r_{i}$ and $g=\sum\limits_{k\in G}r'_{k}$ be elements of $R$ with $G$ a torsion-free Abelian group and $r_{i},r'_{i}\in R_{i}$ for all $i$. Then $fg\in I$ if and only if $r_{i}r'_{k}\in I$ for all $i,k\in G$.
\end{corollary}

\begin{proof}
If $fg\in I$, then $g\in I:_{R}Rf$. By
Theorem~\ref{Theorem vii}, $I:_{R}Rf$ is a graded ideal. Thus, $fr'_{k}\in I$ for all $k\in G$. It follows that each $r_{i}r'_{k}\in I$, since $I$ is a graded ideal. The reverse implication is obvious.
\end{proof}

\begin{corollary}\label{Theorem III}
Let $f=\sum\limits_{i\in G}r_{i}$ and $g=\sum\limits_{k\in G}r'_{k}$ be elements of a $G$-graded ring $R=\bigoplus\limits_{i\in G}R_{i}$ with $G$ a torsion-free Abelian group and $r_{i},r'_{i}\in R_{i}$ for all $i$. Then $fg$ is nilpotent if and only if $r_{i}r'_{k}$ is nilpotent for all $i,k\in G$.
\end{corollary}

\begin{proof} It follows from  Corollary \ref{Theorem IV}.
\end{proof}

All of the above four results can be easily generalized to the setting of $M$-graded rings where $M$ is a totally ordered monoid with the cancellation property. 

\begin{example}(An ideal whose annihilator is not graded)\label{Ex:Deligne}
In order to find an ideal in a $G$-graded ring (with $G$ a torsion-free Abelian group) whose annihilator is not a graded ideal, the ideal should not be a graded ideal and by Corollary~\ref{Coro xii}, the ring must not be reduced. The question of whether such an ideal exists is highly interesting in the light of Theorem~\ref{th T-O} and its consequences.
Finding a concrete example of such an ideal is not an easy task, but fortunately Pierre Deligne has provided us with an example: Let $k$ be a field (or, an integral domain) and let $R$ be the polynomial ring $k[x_{1},x_{2},x_{3},x_{4}]$ modulo the ideal $I=(x_{1}x_{3},x_{2}x_{4}, x_{1}x_{4}+x_{2}x_{3})$. Then in the $\mathbb{N}$-graded polynomial ring $S:=R[T]$ with $\deg(T)=1$ we have $(a_{1}T+a_{2})(a_{3}T+a_{4})=0$ where $a_{i}:=x_{i}+I$. Thus, $a_{1}T+a_{2}\in\Ann_{S}(a_{3}T+a_{4})$, but the annihilator does not contain $a_{1}T$ nor $a_{2}$, because $x_{1}x_{4}, x_{2}x_{3}\notin I$. Hence, $\Ann_{S}(a_{3}T+a_{4})$ is not a graded ideal of $S$.
\end{example}

Recall that if $I$ is a nonempty subset of a ring $R$ and $M$ is a monoid, then by $I[M]$ we mean the set of all $\sum\limits_{m\in M}r_{m}\epsilon_{m}\in R[M]$ such that $r_{m}\in I$ for all $m\in M$. If $I$ is an ideal of $R$ then $I[M]$ is a graded ideal of $R[M]$. In fact, $I[M]$ is the extension of $I$ under the canonical ring map $R\rightarrow R[M]$ that is given by $r\mapsto r\epsilon_{0}$. It is also clear that the contraction of $I[M]$ under this map equals $I$. We have then the following result which codifies some important properties of monoid-rings. 

\begin{theorem}\label{Corollary onbes-fifteen} Let $R$ be a ring and $M$ a totally ordered cancellative monoid. Then the following assertions hold. \\ 
$\mathbf{(i)}$ A nonempty subset  $\mathfrak{p}$ of $R$ is a prime ideal of $R$ if and only if  $\mathfrak{p}[M]$ is a prime ideal of $R[M]$. In particular, $R$ is a domain if and only if $R[M]$ is a domain. \\
$\mathbf{(ii)}$ The minimal primes of $R[M]$ are precisely of the form $\mathfrak{p}[M]$ where $\mathfrak{p}$ is a minimal prime of $R$. \\
$\mathbf{(iii)}$ The nilradical of $R[M]$ is of the form $\mathfrak{N}[M]$ where $\mathfrak{N}$ is the nilradical of $R$. \\
$\mathbf{(iv)}$ $R$ is a reduced ring if and only if $R[M]$ is reduced. \\
$\mathbf{(v)}$ $Z(R[M])\subseteq\bigcup\limits_{\mathfrak{p}\in\Spec(R)}
\mathfrak{p}[M]$. \\
$\mathbf{(vi)}$ If $R$ is a zero-dimensional ring, then $Z(R[M])=\bigcup\limits_{\mathfrak{p}\in\Spec(R)}
\mathfrak{p}[M]$. \\
$\mathbf{(vii)}$ For given $f=\sum\limits_{x\in M}r_{x}\epsilon_{x}$ and $g=\sum\limits_{y\in M}r'_{y}\epsilon_{y}$ in $R[M]$, then $fg$ is nilpotent if and only if $r_{x}r'_{y}$ is nilpotent for all $x,y\in M$. \\
$\mathbf{(viii)}$ An ideal $I$ of $R$ is faithful if and only if $I[M]$ is a faithful ideal of $R[M]$. 
\end{theorem} 

\begin{proof} (i): If $\mathfrak{p}$ is a prime ideal of $R$, then by Corollary \ref{Coro onalti 16}, $\mathfrak{p}[M]$ is a prime ideal of $R$. Conversely, it can be easily seen that for a nonempty subset $\mathfrak{p}\subseteq R$, if $\mathfrak{p}[M]$ is a (prime) ideal of $R[M]$ then $\mathfrak{p}$ is a (prime) ideal of $R$. \\
(ii): If $\mathfrak{p}$ is a minimal prime ideal of $R$, then by (ii), $\mathfrak{p}[M]$ is a prime ideal of $R$. Let $Q$ be a minimal prime of $R[M]$ with $Q\subseteq\mathfrak{p}[M]$. Then $Q\cap R\subseteq\mathfrak{p}[M]\cap R=\mathfrak{p}$ and so $Q\cap R=\mathfrak{p}$. This yields that $Q=\mathfrak{p}[M]$. Now let $P$ be a minimal prime of $R[M]$. Then $\mathfrak{p}:=P\cap R$ is a minimal prime of $R$ and it is easy to see that $\mathfrak{p}[M]\subseteq P$. By the $M$-graded version of Corollary \ref{Coro 14-ondort}, $P$ is a graded ideal of $R[M]$. Hence, to see the inclusion $P\subseteq\mathfrak{p}[M]$, it suffices to check it for homogeneous elements of $P$. Take $r\epsilon_{m}\in P$ where $r\in R$ and $m\in M$. We may write $r\epsilon_{m}=(r\epsilon_{0})\epsilon_{m}$. But $\epsilon_{m}\notin P$, because $\epsilon_{m}$ is a non-zero-divisor of $R[M]$ while $P$ is a minimal prime of $R[M]$ and hence it is contained in the set of zero-divisors of $R[M]$. It follows that $r=r\epsilon_{0}\in P\cap R=\mathfrak{p}$. Hence, $r\epsilon_{m}\in\mathfrak{p}[M]$. \\
(iii): It is clear that $\mathfrak{N}[M]$ is contained in the nilradical of $R[M]$. By the $M$-graded version of Corollary \ref{Coro 14-ondort}, the nilradical of $R[M]$ is a graded ideal. Hence, to prove the reverse inclusion, it suffices to check it for homogeneous nilpotents. If  $r\epsilon_{m}=(r\epsilon_{0})\epsilon_{m}$ is nilpotent with $r\in R$ and $m\in M$, then $r=r\epsilon_{0}$ is nilpotent because $\epsilon_{m}$ is a non-zero-divisor. Thus $r\epsilon_{m}\in\mathfrak{N}[M]$. \\
(iv): It is clear from (iii).  \\
(v): If $f=\sum\limits_{m\in M}r_{m}\epsilon_{m}$ is a zero-divisor of $R[M]$ then by Corollary \ref{Coro 5-bes-Kap}, there exists a nonzero $a\in R$ such that $af=0$. It follows that $ar_{m}=0$ for all $m\in M$. Consider the ideal $I=(r_{m}: m\in M)$ of $R$.
Then  $I$ is a proper ideal of $R$, because $a\neq0$. Thus, there exists a prime ideal $\mathfrak{p}$ of $R$ such that $I\subseteq\mathfrak{p}$ and so $f\in\mathfrak{p}[M]$. \\
(vi): If $\mathfrak{p}$ is a prime ideal of $R$ then by (ii), $\mathfrak{p}[M]$ is a minimal prime of $R[M]$. But it is well-known that every minimal prime of a ring is contained in the set of zero-divisors. So $\bigcup\limits_{\mathfrak{p}\in\Spec(R)}
\mathfrak{p}[M]\subseteq Z(R[M])$. The reverse inclusion follows from 
(v). \\
(vii): By the $M$-graded version of Corollary \ref{Theorem III}, $fg$ is nilpotent if and only if every $(r_{x}\epsilon_{x})(r'_{y}\epsilon_{y})=r_{x}r'_{y}\epsilon_{x+y}$ is nilpotent. But $\epsilon_{x+y}$ is a non-zero-divisor. Hence, $fg$ is nilpotent if and only if each $r_{x}r'_{y}$ is nilpotent. \\
(viii): Assume $I$ is faithful. If $I[M]$ is an unfaithful ideal of $R[M]$, then by Corollary \ref{Coro 6-altin}, $(r\epsilon_{m})I[M]=0$ for some $0\neq r\in R$ and some $m\in M$. But $\epsilon_{m}$ is a non-zero-divisor, and hence $r\in\Ann(I)=0$ which is a contradiction. The reverse implication is clear. 
\end{proof}

The above theorem, in particular, can be applied to the ring of polynomials as well as to the ring of Laurent polynomials with any (finite or infinite) number of variables. Theorem \ref{Corollary onbes-fifteen}(vii) generalizes  Armendariz' result \cite[Lemma 1]{Armendariz} and also  \cite[Proposition 2.1]{Antoine} to the more general setting of monoid-rings.

\begin{example} Theorem \ref{Corollary onbes-fifteen} does not hold in general. For example, let $G$ be an Abelian group which is not a totally ordered group. So $G$ has a nonzero element $g$ of finite order $n\geqslant2$. Let $p$ be a prime number which divides $n$. Then $x:=(n/p)g$ is a nonzero element of order $p$. Let $R$ be an integral domain (or, a reduced ring) of characteristic $p$. Then the group-ring $R[G]$ is not reduced, because the Frobenius endomorphism yields that $(\epsilon_{0}-\epsilon_{x})^{p}=
(\epsilon_{0})^{p}-(\epsilon_{x})^{p}=\epsilon_{0}-\epsilon_{px}=0$, but $\epsilon_{0}\neq\epsilon_{x}$. Furthermore, $\epsilon_{0}-\epsilon_{x}$ is a zero-divisor of $R[G]$ but it is not contained in  
$\bigcup\limits_{\mathfrak{p}\in\Spec(R)}
\mathfrak{p}[G]$.
\end{example}

In the following result we completely characterize units in group-rings:  

\begin{theorem}\label{coro ET new 20} If $R$ is a ring and $G$ is a torsion-free Abelian group, then the units of $R[G]$ are precisely of the form $f=\sum\limits_{x\in G}r_{x}\epsilon_{x}$ with $(r_{x}: x \in G)=R$ and $r_{x}r_{y}$ is nilpotent for all $x\neq y$. 
\end{theorem}

\begin{proof} If $f=\sum\limits_{x\in G}r_{x}\epsilon_{x}$ is invertible in $R[G]$, then by Theorem \ref{Theorem ix}, we have $(r_{x}\epsilon_{x} : x\in G)=R[G]$ and $(r_{x}\epsilon_{x})(r_{y}\epsilon_{y})=
r_{x}r_{y}\epsilon_{x+y}$ is nilpotent for all $x\neq y$. But $\epsilon_{x+y}$ is a non-zero-divisor of $R[G]$, because $G$ has the cancellation property. This shows that $r_{x}r_{y}$ is nilpotent for all $x\neq y$. We also have $(r_{x}: x \in G)=R$. If not, then $(r_{x}: x \in G)\subseteq\mathfrak{p}$ for some prime ideal $\mathfrak{p}$ of $R$. Then $R[G]=(r_{x}\epsilon_{x} : x\in G)\subseteq\mathfrak{p}[G]$. But this is a contradiction, because by Theorem \ref{Corollary onbes-fifteen}(i), $\mathfrak{p}[G]$ is a prime ideal of $R[G]$. Now we prove the reverse implication. Since $r_{x}r_{y}$ is nilpotent, thus $(r_{x}\epsilon_{x})(r_{y}\epsilon_{y})$ is also nilpotent for all $x\neq y$. Next, we show that $(r_{x}\epsilon_{x} : x\in G)=R[G]$. If not, then $(r_{x}\epsilon_{x} : x\in G)\subseteq P$ for some prime ideal $P$ of $R[G]$. Since $G$ is a group, $r_{x}\epsilon_{0}=(r_{x}\epsilon_{x})\epsilon_{-x}\in P$ for all $x\in G$. But we have $1=\sum\limits_{x\in G}s_{x}r_{x}$ with $s_{x}\in R$ for all $x$. It follows that $\epsilon_{0}=\sum\limits_{x\in G}s_{x}r_{x}\epsilon_{0}=\sum\limits_{x\in G}(s_{x}\epsilon_{0})(r_{x}\epsilon_{0})\in P$ which is a contradiction. Then by Lemma \ref{Lemma 3-easy}, $f=\sum\limits_{x\in G}r_{x}\epsilon_{x}$ is invertible in $R[G]$. 
\end{proof}

If $f=\sum\limits_{x\in G}r_{x}\epsilon_{x}$ is an invertible element of $R[M]$ with $M$ a totally ordered cancellative monoid then just like the above result, $(r_{x}: x \in M)=R$ and $r_{x}r_{y}$ is nilpotent for all $x\neq y$. But it is important to note that if $M$ is not a group, then the reverse implication is not true. Indeed, in this case, there exists some $x\in M$ that is not invertible in $M$. Then the element $f=\epsilon_{x}$ satisfies the condition of the above result, but it is not invertible in $R[M]$.

\begin{remark}\label{Remark 1-bir} Here we describe an interesting connection between the general graded rings and monoid-rings (which is so useful in providing simple alternative proofs or discovering new observations in graded ring theory): Let $R=\bigoplus\limits_{m\in M}R_{m}$ be an $M$-graded ring with $1\in R_0$ where 0 is the identity element of $M$ (e.g. $M$ is a cancellative monoid). Then $R$ can be naturally viewed as a graded subring of the monoid-ring $R[M]$ via the embedding $\theta:R\rightarrow R[M]$ which sends each $r_{m}\in R_{m}$ into $r_{m}\epsilon_{m}$ for all $m\in M$. In particular, every $\mathbb{N}$-graded ring $R=\bigoplus\limits_{n\geqslant0}R_{n}$ can be viewed as a graded subring of the polynomial ring $R[x]$ via the embedding $R\rightarrow R[x]$ which sends each $r\in R_{n}$ into $rx^{n}$ for all $n\geqslant0$, and every $\mathbb{Z}$-graded ring $R=\bigoplus\limits_{n\in\mathbb{Z}}R_{n}$ can be viewed as a graded subring of the ring of Laurent polynomials $R[x,x^{-1}]$ via the embedding $R\rightarrow R[x,x^{-1}]$ which sends each $r\in R_{n}$ into $rx^{n}$ for all $n\in\mathbb{Z}$. The ring $R[x,x^{-1}]$ also can be viewed as a subring of the formal power series ring $R[[x]]$. Indeed, using the universal property of polynomial rings, we obtain an injective morphism of $R$-algebras $R[x]\rightarrow R[[x]]$ given by $x\mapsto 1-x$. But $1-x$ is invertible in $R[[x]]$ (remember that an element of $R[[x]]$ is invertible if and only if its constant term is invertible in $R$). Thus by the universal property of localizations, we get an injective morphism of $R$-algebras $\psi:R[x,x^{-1}]\rightarrow R[[x]]$ with $\psi(x)=1-x$ and $\psi(x^{-1})=(1-x)^{-1}=\sum\limits_{k\geqslant0}x^{k}$. Finally, if $M$ is a totally ordered monoid with the cancellation property then an ideal $I$ of $R$ is faithful if and only if its extension under the above embedding $\theta:R\rightarrow R[M]$ is a faithful ideal of $R[M]$. Indeed, assume $I$ is faithful. If the extension ideal $I^{e}$ is unfaithful then by Corollary \ref{Coro 5-bes-Kap}, $rI^{e}=0$ for some $0\neq r\in R$. If $f=\sum\limits_{x\in M}r_{x}\in I$ then $\theta(f)=\sum\limits_{x\in M}r_{x}\epsilon_{x}\in I^{e}$ and so $\sum\limits_{x\in M}rr_{x}\epsilon_{x}=0$. Since $M$ has the cancellation property, $rr_{x}\epsilon_{x}=0$ and so $rr_{x}=0$ for all $x\in M$. This shows that $rf=0$. Hence, $r\in\Ann_{R}(I)=0$ which is a contradiction. The reverse implication is proved similarly, by applying Corollary \ref{Coro 6-altin} instead of Corollary \ref{Coro 5-bes-Kap}.
\end{remark} 

We have the following theorem which is obtained in the light of the above results. 
 
\begin{theorem}\label{Bergman's th charact} For an Abelian group $G$ the following assertions are equivalent: \\
$\mathbf{(i)}$ $G$ is a totally ordered group.  \\
$\mathbf{(ii)}$ $G$ is a torsion-free group. \\
$\mathbf{(iii)}$ The nilradical of every $G$-graded ring is a graded ideal. \\
$\mathbf{(iv)}$ For each prime number $p$, the nilradical of the group-ring $(\mathbb{Z}/p)[G]$ is a graded ideal. \\
$\mathbf{(v)}$ The Jacobson radical of every $G$-graded ring is a graded ideal. \\
$\mathbf{(vi)}$ For each prime number $p$, the Jacobson radical of the group-ring $(\mathbb{Z}/p)[G]$ is a graded ideal. \\
$\mathbf{(vii)}$ The idempotents of every $G$-graded ring is contained in its base subring. \\
$\mathbf{(viii)}$ For any ring $R$, every idempotent of the group-ring $R[G]$ is contained in $R$. \\
$\mathbf{(ix)}$ For any field $K$, the group-ring $K[G]$ has no nontrivial idempotents. \\
$\mathbf{(x)}$ The group-ring $\mathbb{Q}[G]$ has no nontrivial idempotents. \\
$\mathbf{(xi)}$ For any integral domain $R$, the group-ring $R[G]$ is an integral domain. \\
$\mathbf{(xii)}$ The group-ring $\mathbb{Q}[G]$ 
is an integral domain. \\
$\mathbf{(xiii)}$ For any integral domain $R$, the units of the group-ring $R[G]$ are precisely of the form $r\epsilon_{g}$ with $r$ is invertible in $R$ and $g\in G$. \\
$\mathbf{(xiv)}$ The units of the group-ring $\mathbb{Q}[G]$ are precisely of the form  $r\epsilon_{g}$ with $0\neq r\in\mathbb{Q}$ and $g\in G$. 
\end{theorem}

\begin{proof} (i)$\Leftrightarrow$(ii): This is well-known. \\
(ii)$\Rightarrow$(iii): See Corollary \ref{Coro 14-ondort}. \\
(iii)$\Rightarrow$(iv): There is nothing to prove. \\
(iv)$\Rightarrow$(ii): If $G$ is not torsion-free, then it has a nonzero element $h\in G$ of finite order $n\geqslant2$. Let $p$ be a prime divisor of $n$, and note that $g:=(n/p)h\in G$ is a nonzero element of order $p$. Clearly in the group-ring $\mathbb{Z}_{p}[G]$ the element $\epsilon_{0}-\epsilon_{g}$ is nilpotent, because the Frobenius endomorphism yields that $(\epsilon_{0}-\epsilon_{g})^{p}=(\epsilon_{0})^{p}-(\epsilon_{g})^{p}
=\epsilon_{0}-\epsilon_{pg}=\epsilon_{0}-\epsilon_{0}=0$. But $\epsilon_{0}$ (as well as $\epsilon_{g}$) is not nilpotent which is a contradiction. \\
(ii)$\Rightarrow$(v): This is Theorem \ref{Theorem II}. \\
(v)$\Rightarrow$(vi): It is clear. \\
(vi)$\Rightarrow$(ii): It is proved exactly like the implication (iv)$\Rightarrow$(ii). \\
(ii)$\Rightarrow$(vii): This is Theorem \ref{Lemma T-K}. \\
(vii)$\Rightarrow$(viii)$\Rightarrow$(ix)$\Rightarrow$(x): All implications obvious. \\
(x)$\Rightarrow$(ii): If $G$ is not torsion-free, then it has a nonzero element $g\in G$ of finite order $n\geqslant2$. Then the element $f:=
(1/n)\sum\limits_{k=0}^{n-1}\epsilon_{kg}=(1/n)
\big(\epsilon_{0}+\epsilon_{g}+\cdots+\epsilon_{(n-1)g}\big)$ of the group-ring $\mathbb{Q}[G]$ is a nontrivial idempotent. Indeed, $f^{2}=(1/n^{2})\sum\limits_{0\leqslant k,d\leqslant n-1}\epsilon_{(k+d)g}=
n(1/n^{2})\sum\limits_{k=0}^{n-1}\epsilon_{kg}=f$. \\
(ii)$\Rightarrow$(xi): Suppose $(a\epsilon_{x})(b\epsilon_{y})=0$ for some $a,b\in R$ and some $x,y\in G$. Then $ab\epsilon_{x+y}=0$ and so $ab=0$. This yields that $a=0$ or $b=0$. Thus by Lemma \ref{Lemma 2-iki}, the zero ideal of $R[G]$ is a prime ideal. \\
(xi)$\Rightarrow$(xii)$\Rightarrow$(x): Clear. \\
(ii)$\Rightarrow$(xiii): It is clear from Corollary \ref{coro ET new 20}. \\
(xiii)$\Rightarrow$(xiv): Obvious. \\
(xiv)$\Rightarrow$(ii): If $G$ is not torsion-free, then it has a nonzero element $g\in G$ of finite order $n\geqslant2$. If $n=2$ then $2\epsilon_{0}+\epsilon_{g}$ is invertible in $\mathbb{Q}[G]$, because $(2\epsilon_{0}+\epsilon_{g})\cdot
\big((2/3)\epsilon_{0}-(1/3)\epsilon_{g}\big)
=\epsilon_{0}=1$. If $n\geqslant3$ then $f:=(1/n)\sum\limits_{k=0}^{n-1}\epsilon_{kg}=
(1/n)(\epsilon_{0}+\cdots+\epsilon_{(n-1)g})$ is an idempotent and so $1-2f=\big((n-2)/n\big)\epsilon_{0}-(2/n)
\sum\limits_{k=1}^{n-1}\epsilon_{kg}$ is invertible in $\mathbb{Q}[G]$. But $2\epsilon_{0}+\epsilon_{g}$ and $1-2f$ are not homogeneous which contradicts the hypothesis. 
\end{proof} 

We know that the nilradical of any ring is contained in its Jacobson radical. If the nilradical and the Jacobson radical of a ring are the same, then we call it a quasi-Jacobson ring. Note that this notion generalizes the classical ``Jacobson ring'' notion. Recall that a ring is called a \emph{Jacobson ring} if every prime ideal is an intersection of some maximal ideals. In the literature, Jacobson rings are also called Hilbert rings. One can show that a ring $R$ is Jacobson if and only if each quotient ring $R/I$ is quasi-Jacobson where $I$ is an ideal of $R$. It can be also seen that a ring $R$ is quasi-Jacobson if and only if its maximal spectrum $\Max(R)$ is Zariski dense in $\Spec(R)$. In this regard, we have the following result. 

\begin{theorem}\label{Lemma 3 NZD} If a $G$-graded ring $R=\bigoplus\limits_{x\in G}R_{x}$ with $G$ a torsion-free Abelian group has a non-zero-divisor homogeneous element of nonzero degree, then $R$ is quasi-Jacobson.
\end{theorem}

\begin{proof} It suffices to show that $\mathfrak{J}(R)\subseteq\mathfrak{N}(R)$. Take $f=\sum\limits_{x\in G}r_{x}\in\mathfrak{J}(R)$ where $r_{x}\in R_{x}$ for all $x\in G$. Let $a\in R_{y}$ be a non-zero-divisor of $R$ for some $0\neq y\in G$.  
Since $y$ is of infinite order (because $G$ is torsion-free) and since $\Supp(f)$ is a finite set, we can find a natural number $N\geqslant1$ such that $x+Ny\neq0$ for all $x\in\Supp(f)$. We know that $1+fa^{N}$ is invertible in $R$. Note that the homogeneous component of degree zero of the element $1+fa^{N}$ is 1. Then using Theorem \ref{Theorem ix} and the fact that $a^{N}$ is a non-zero-divisor of $R$, we have $r_{x}a^{N}$ and so $r_{x}$ are nilpotent for all $x\in\Supp(f)$. Hence, $f=\sum\limits_{x\in\Supp(f)}r_{x}$ is nilpotent.  
\end{proof}

\begin{corollary}\label{Coro 2 Gro} Let $R=\bigoplus\limits_{m\in M}R_{m}$ be an $M$-graded ring where $M$ is a totally ordered monoid with the cancellation property. If $R$ has a non-zero-divisor homogeneous element of nonzero degree, then $R$ is quasi-Jacobson. 
\end{corollary} 

\begin{proof} It follows from Theorem \ref{Lemma 3 NZD} by passing to the $G$-grading where $G$ is the Grothendieck group of $M$ which is a torsion-free group.  
\end{proof}

\begin{corollary}\label{Coro 3 nice} Let $R$ be a ring and $M$ a nonzero totally ordered monoid with the cancellation property. Then the monoid-ring $R[M]$ is quasi-Jacobson. 
\end{corollary} 

\begin{proof} We may assume $R$ is a nonzero ring (the assertion is clear for the zero ring). Since $M$ has the cancellation property, $\epsilon_{m}$ is a non-zero-divisor of $R[M]$ for all $m\in M$. Since $M$ is nonzero, $R[M]$ has a non-zero-divisor homogeneous element of nonzero degree.
Then apply the above corollary.
\end{proof}

By the above result, for any ring $R$, the polynomial ring $R[x]$ is quasi-Jacobson. If $G$ is a nonzero torsion-free group, then the group-ring $R[G]$ is quasi-Jacobson. In particular, the ring of Laurent polynomials $R[x,x^{-1}]$ is quasi-Jacobson.

\end{document}